\documentclass[10pt]{article}
\usepackage{geometry}

\usepackage{comment,hyperref,slashed,tensor}

\usepackage{amsmath}
\usepackage[T1]{fontenc}
\usepackage[utf8]{inputenc}

\usepackage{ mathrsfs }
\usepackage{amsmath, amsfonts, mathtools, amsthm, amssymb}

\mathtoolsset{showonlyrefs}
\usepackage{todonotes, tikz}
\usepackage{hyperref,  color}

\newtheorem{theorem}{Theorem}[section]
\newtheorem{lemma}[theorem]{Lemma}
\newtheorem{proposition}[theorem]{Proposition}

\newtheorem{conjecture}[theorem]{Conjecture}

\newtheorem{cor}[theorem]{Corollary}

\newtheorem{definition}[theorem]{Definition}

\newtheorem{problem}[theorem]{Problem}
\newtheorem{question}[theorem]{Question}
\newtheorem{remark}[theorem]{Remark}
\newtheorem{hypothesis}[theorem]{Hypothesis}

\newcommand{\be}{\begin{equation}}
\newcommand{\ee}{\end{equation}}

\def\XXint#1#2#3{{\setbox0=\hbox{$#1{#2#3}{\int}$ }
\vcenter{\hbox{$#2#3$ }}\kern-.6\wd0}}

\title{Dynamics of Ideal Fluid Flows}
\author{Tarek M. Elgindi}
\date{\today}

\begin{document}

\maketitle

\begin{abstract}
We will discuss various aspects of the incompressible Euler equation. We will discuss, in particular, problems related to the least action principle, the existence of special solutions, the problem of solvability, singularity formation, and asymptotic behavior. \end{abstract}
\tableofcontents
\section{Introduction}
Ideal fluids are perhaps the most basic of physical media, their defining property being that they occupy space. As the fluid moves around, no region can increase or decrease its total measure. This is called the measure preserving or incompressibility condition. The incompressibility condition can be seen as a constraint on what type of motion we allow for in a fluid. Since we believe that all physical motion occurs in such a way as to minimize some energy, even if we do not know \emph{a-priori} how to quantify it, we can derive the equations of motion of an ideal fluid from the \emph{principle of least action} while enforcing the incompressibility constraint. The following gives a simple illustration of the type of question that one could ask in this direction. 

\begin{problem}\label{OptimalProblem}
Imagine that all of $\mathbb{R}^d$ is filled with an incompressible fluid. Assume that we isolate in our minds some bounded region $\Omega\subset\mathbb{R}^d$ and we wish to transport it to its translate $\Omega+e_1$ by time $T$, while {\bf minimizing} the total kinetic energy. How do we complete the task?
\end{problem}

We will come to a more detailed discussion of this problem later in Section \ref{Optimal}, but the immediate reaction of someone who hears this question is to say: you should translate $\Omega$ rigidly, of course! This is indeed what optimal transport would say in such a situation.  We should realize, however, that such a motion in a part of the fluid immediately affects the rest of the fluid (because the fluid to the right has to move out of the way!). In other words, we cannot simply prescribe the motion of the fluid inside $\Omega$; we must prescribe the motion in all of $\mathbb{R}^d.$ This is the central problem with ideal fluids: the constraint of volume preservation implies that when one region of fluid moves, everything else moves in response.  This effect we are describing is called \emph{non-locality} and it plays a central role in the dynamics of ideal fluids. As trivial as Problem \ref{OptimalProblem} sounds, the author is not aware of any work in which this problem is studied or even discussed. 

Once we actually carry out the method of least action, taking into account measure preservation, we get a system of partial differential equations describing the motion. This system was first derived by Euler in 1757, and it serves as a quite reasonable model for the motion of a variety of substances, including air and water, while also describing a fundamental mathematical object. Once we have the equations of motion, there are a number of questions we could ask. A first interesting problem is to study the existence of particular solutions to the equation of motion. It turns out that the equations of motion are consistent with many of the types of fluid motion we are familiar with from our daily experience: laminar flow, vortices, and vortex rings. The emergence of such \emph{special solutions} is one of the most remarkable aspects of fluids: what are they, and why do they always seem to appear? Examples of such special solutions are steady solutions, translating solutions, and time-periodic solutions. The relevant problem can be stated somewhat informally as follows:
\begin{problem}
Give a complete description of the set of special solutions to the Euler equation and explain why they emerge. 
\end{problem}
\noindent The two parts of this problem, while deeply related, are of a fundamentally different nature. The first one appears to be essentially static and aspects of it are amenable to the techniques of nonlinear functional analysis. The second question is dynamic and requires the discovery of perturbative and non-perturbative mechanisms that lead such special solutions to attract others. A huge swath of the literature, both pure and applied, on ideal fluids has been devoted to this problem. One of the concrete directions under this problem is the question of long-time relaxation for 2d solutions, which will be discussed in some detail below.  


Another interesting problem relates to solvability
. 
If we view fluid motion as similar to the geodesic motion of a free particle on a surface (the surface here being the set of volume preserving diffeomorphisms), there are two fundamental questions we can ask: do geodesic paths continue forever, and can all points on (connected components of) the surface be connected by a geodesic path? The first one can be seen as a question of global solvability for an initial value problem, while the second can be seen as a question of global solvability for a boundary value problem.  This can be stated formally as follows. 
\begin{problem}
Do solutions to the Euler equation exist for all time? Can any two attainable fluid configurations be connected by an Euler solution?
\end{problem}
\noindent In dimension $d=2$, we have global existence for classical solutions. When $d\geq 3,$ global existence is known to break down in the setting of classical solutions and is likely to break down no matter what regularity is assumed. This is due to the formation of \emph{finite-time singularities}. As for the solvability of the boundary value problem, i.e. the geodesic completeness of the space SDiff, it is known that \emph{optimal} paths exist locally for $d\geq 2$  \cite{EbinMarsden} but need not exist globally when $d\geq 3$ \cite{Shnirelman}. It is also known that arbitrary configurations can be connected by so-called generalized solutions \cite{Brenier}. However, whether attainable configurations can be connected by non-optimal paths when $d\geq 2$ is unknown and unclear. One thing that is clear is that the two questions above represent two completely different approaches to the modeling of fluids. Global solvability for the Cauchy problem, which is a question about the global predictive power of the equation, need have nothing to do with the global solvability for the two point problem, which we could say is a question about the ability of the Euler equation to model observed fluid motion.  

In the coming sections, we will begin with a brief discussion of the least action principle and then move on to discuss steady solutions, mixing, and the conjectured behavior of 2d Euler solutions in infinite time. We will then move on to discuss the solvability problem and the formation of singularities in finite time. Along the way, we will propose problems of various difficulty, some of which could bring our loftier goals a little bit closer. 



\section{Optimal Motion}\label{Optimal}
As a way to give a quick derivation of the Euler equation and to motivate interest in a lamentably understudied direction, let us now discuss Problem \ref{OptimalProblem} in some more detail.  By rescaling, we may set $T=1$. We may also generalize the problem to transporting some distribution $h_0$ to some volume preserving rearrangment of it, $h_1$ \footnote{Note that we are going to insist here that the measure we use is the Lebesgue measure, but this can of course be relaxed.}.  Now we wish to minimize:
\[J(u)=\int_0^1\int_{\mathbb{R}^d} |u(x,t)|^2dxdt.\] Over all $u$ that transport $h_0$ to $h_1,$ keeping volume preservation throughout $\mathbb{R}^d$. This leads to the following constraints. 
\begin{equation}\label{FuEqn}\partial_t f_u+u\cdot\nabla f_u=0,\end{equation}
\[f_u|_{t=0}=h_0,\qquad f_u|_{t=1}=h_1,\]
\[\text{div}(u)=0,\] the final constraint being the constraint of volume preservation\footnote{This is because $\text{div}(u)=0$ is equivalent to $\int_{\partial B} u\cdot n=0$ for any ball $B,$ which is equivalent to volume preservation.}. An optimizer would have to satisfy:
\begin{equation}\label{FirstVariation}\int_0^1\int u\cdot v=0,\end{equation} for all admissible variations $v.$ To find those admissible $v,$ we vary the constraints. The divergence-free constraint is linear. Therefore, varying the constraints in the direction of some divergence-free $v,$ and setting $g:=\partial_v f_u,$ gives that the solution to 
\[\partial_t g+u\cdot\nabla g+v\cdot\nabla f_u=0,\] must satisfy
\[g|_{t=0}=g|_{t=1}=0.\] Denote by $\Phi_u$ the solution to the ODE:
\[\frac{d}{dt}\Phi_u(x,t)=u(\Phi_u(x,t),t),\] with $\Phi_u(x,0)=x$. Since $f_u$ satisfies \eqref{FuEqn}, we find that \[\frac{d}{dt}(g\circ\Phi_u)=-\Phi_u^*v\cdot\nabla h_0,\] where we define $\Phi_u^*$ by
\[\Phi_u^*h=(\nabla\Phi_u)^{-1}h\circ\Phi_u.\]
In particular, to ensure that $g$ is zero initially and terminally, the constraint is that
\begin{equation}\label{ConstraintEquation}\int_0^1\Phi_u^*v\cdot \nabla h_0=0.\end{equation} 
As a particular case, we may consider all divergence free $v$ for which  
\[\int_0^1\Phi_u^*v=0,\] which is the set of all divergence free $v$ for which 
\[v= \Big((\nabla\Phi_u)\frac{d}{dt}\eta\Big)\circ\Phi_u^{-1},\] for some $\eta|_{t=0}=\eta|_{t=1}=0$ with $\eta$ divergence-free. 
We thus get that the optimizer $u$ must satisfy:
\[\int_0^1\int \frac{d}{dt}((\nabla\Phi_u)^{t}(u\circ\Phi_u)) \cdot \eta=0,\] for all divergence-free $\eta,$ which is equivalent to the statement:
\begin{equation}\label{Weberish}\mathbb{P}((\nabla\Phi_u)^{t}(u\circ\Phi_u))=u_0,\end{equation} for all $t\in [0,T].$ Here $\mathbb{P}$ is the Leray projector. This is equivalent to the so-called Weber formula for solutions to the Euler equation \cite{Peter}, which is equivalent to the velocity formulation:
\begin{equation}\label{VelocityFormulation}
\partial_t u+ u\cdot\nabla u+\nabla p=0,
\end{equation}
\begin{equation}\label{divfree}
\text{div}(u)=0.
\end{equation}
 Taking the curl of \eqref{Weberish}, in the three dimensional case, a direct calculation gives that
\[\Phi_u^*\text{curl}(u)=\text{curl}(u_0).\] This is simply the statement that $\text{curl}(u)$ is Lie transported by $u.$ In particular, if we set for an optimizer $u$ of our Problem,
\[\omega=\text{curl}(u),\] we find that $\omega$ satisfies:
\begin{equation}\label{VorticityEquation}\partial_t\omega+[u,\omega]=0,\end{equation}
\begin{equation}\label{BSLaw}u=\text{curl}^{-1}(\omega),\end{equation} where $\text{curl}^{-1}$ is the uniquely defined operator on divergence-free vector fields and \[[u,\omega]=\partial_u \omega-\partial_\omega u.\]  This is the well-known vorticity formulation of the Euler equation.
Restating what we have found so far:
\begin{cor}
Any optimizer to Problem \eqref{OptimalProblem} satisfies the Euler equation. 
\end{cor}
Notice, however, that Problem \ref{OptimalProblem} actually satisfies many more constraints. In particular, we did not use the full freedom on $v$ given by the constraint equation \eqref{ConstraintEquation}. This will be discussed in much greater detail elsewhere, but let us record here some of what can be said in the case $d=2.$ 
\begin{theorem}
When $d=2$ and $\Omega$ is a smooth domain, an optimal solution to Problem \eqref{OptimalProblem} must solve the Euler equation and satisfy that $\omega_0$ is a measure on $\partial\Omega.$ In particular, an optimal solution must be a vortex sheet. When $\Omega=B_1(0),$ the solution is given by an explicit translating vortex sheet solution.
\end{theorem}
\noindent The first part is not difficult to prove. One interesting aspect of the above computations is that they generalize naturally to many different variational problems, all of whose optimizers are solutions to the Euler equation \cite{VariationalPrinciples}.  

\begin{remark}
Let us remark that the usual least action principle for the Euler equation \cite{ArnoldKhesin} given as a two-point boundary value problem readily implies the corollary above. The calculations above are just included for pedagogical purposes and to uncover the second constraint for that particular problem (that the solution is a vortex sheet). Furthermore, while Problem \ref{OptimalProblem} doesn't seem to have been studied before, related variational principles have been considered in \cite{MW,CM,HK}.
\end{remark}
While the classical two-point boundary problem for Euler flows may be difficult to solve in general \cite{Shnirelman, Brenier, MP}, it is possible that more relaxed problems such as Problem \ref{OptimalProblem}  could enjoy a good existence theory. 


\section{Steady Solutions} 
As in any dynamical system, the first step to understanding the dynamics is to study the steady solutions and those nearby. Let us start with a look at the three dimensional case. 
Steady solutions to \eqref{VelocityFormulation}-\eqref{divfree} satisfy:
\[u\cdot\nabla u+\nabla p=0,\] which can be re-written as:
\[u\times \omega=\nabla (p+\frac{|u|^2}{2}).\] Calling $H=p+\frac{1}{2}|u|^2,$ we find that:
\[u\times \omega=\nabla H,\] and 
\[u\cdot\nabla H=\omega\cdot\nabla H=0.\] When $H$ is analytic and non-constant, we find that the fluid trajectories are confined to move along the level surfaces of $H,$ which are generically two dimensional tori or cylinders. These are called fibered steady states. In the other case that $H\equiv const$, we must have that 
\[u\times\omega=0.\] These are called Beltrami flows.  Arnold's Structure Theorem \cite{ArnoldKhesin} is simply the statement that, in the analytic class, all steady solutions are either (locally) fibered or are Beltrami flows. From one point of view, the Structure Theorem is quite satisfactory since it seems to say that only two seemingly restrictive things can hold for steady Euler solutions. Unfortunately, not much is known beyond that. In particular, not many fibered solutions are known to exist outside of highly restrictive symmetry classes \cite{Lortz,BrunoLaurence,EncisoJEMS,DEG}. This is the subject of the so-called Grad Conjecture \cite{Grad}. On the side of Beltrami flows, it is known now that they can exhibit a wide array of dynamical phenomena including chaos \cite{EtnyreGhrist,EncisoPeralta,EncisoPeralta2}.   
\subsection{The Two Dimensional Case}
Much more can be said about two-dimensional steady flows. The two-dimensional case can be seen as a special case of the three dimensional case by considering:
\[u(x,y,z)=\begin{pmatrix}
u_1(x,y)\\
u_2(x,y)\\
0
\end{pmatrix},\qquad \omega(x,y,z)=\begin{pmatrix} 
0\\
0\\
\partial_1 u_2(x,y)-\partial_2 u_1(x,y)
\end{pmatrix}.\] 
In the following, we will identify $u$ with just the first two components of $u$ and $\omega$ with just its last component. In particular, $\omega$ will now be viewed as a scalar. 
We then see that the steady vorticity equation is just:
\[\partial_u\omega=0.\]
Since $u$ is divergence-free, we see that we can write:
\[u=\begin{pmatrix}
    -\partial_2\psi\\
    \partial_1\psi,
\end{pmatrix}\]
for some $\psi.$ We thus get that 
\[\partial_1\psi\partial_2\omega-\partial_1\psi\partial_2\omega=0,\]
\[\Delta\psi=\omega.\] This can be written more compactly as:
\begin{equation}\label{SEE}\{\psi,\Delta\psi\}=0,\end{equation} along with the condition that $\psi$ be constant on (connected components of) the boundary of the domain, if it is non-empty. $\psi$ solving \eqref{SEE} is simply the statement that if $\psi$ is constant on some connected set, $\Delta\psi$ must also be constant there. It is not difficult to see that any $\psi$ that is invariant under the Euclidean symmetries, translation, or rotation are automatically steady solutions. In particular, there are two classes of "trivial" steady states: \emph{shear flows} and \emph{radial flows}. There is another class of steady states, which are those $\psi$ satisfying:
\[\Delta\psi=F(\psi),\]
for some suitably differentiable $F.$
It turns out that it is possible to prove what can be seen as a 2d version of Arnold's Structure Theorem in this setting \cite{EHSX}: 
\begin{theorem}\label{Classification}
Let $\Omega$ be an analytic simply connected domain. Assume that $\psi$ is an analytic solution to \eqref{SEE} on $\bar\Omega$ with $\psi$ constant on $\partial\Omega.$ Then, either $\psi$  is radial or there exists $F$ smooth on the interior of the range of $\psi$ for which \begin{equation}\label{SEE2}\Delta\psi=F(\psi).\end{equation}
\end{theorem}
\noindent The theorem is established using symmetrization techniques for semilinear elliptic equations, notably the moving planes method of Aleksandrov. Such techniques have been used for years to prove rigidity theorems for solutions to the Euler equation starting with Serrin's work \cite{Serrin} on overdetermined problems \cite{HN1,HN2,LZ,CDG,CZEW,DN}. Indeed, using the inverse function theorem, it is easy to see that if $\nabla\psi(x_*)\not=0$ and $\psi$ that satisfies \eqref{SEE}, then we can find $F_{x_*}$ for which $\eqref{SEE2}$ is satisfied in a neighborhood of $x_*.$ Using analytic continuation, $F_{x_*}$ can be extended to a single $F$ on a maximal set whose boundary is necessarily contained in the union of $\partial\Omega$ and $\{\nabla\psi=0\}.$ It is not terribly difficult to then argue that Theorem \ref{Classification} not holding implies the existence of some simply connected $\Omega'\subset\Omega$ with smooth boundary for which $\Delta\psi=F(\psi)$ on $\Omega',$ while $\nabla\psi=0$ on $\partial\Omega'.$ If $F$ were Lipschitz continuous, we would be able to apply Serrin's theorem \cite{Serrin}. While $F$ need not be Lipschitz continuous, it is still possible to show that $F'$ is necessarily bounded from below at nonsmooth points in the interior of $\Omega'$. This means that we can still use the maximum principle on the linearization of \eqref{SEE2}. With Theorem \ref{Classification}, the study of analytic steady states to the Euler equation on simply connected domains becomes the same as the study of all analytic solutions to all semi-linear elliptic equations \eqref{SEE2}. It is very important to remark, however, that there are many more solutions to \eqref{SEE} outside of the analytic class. In particular, it was shown in \cite{GPS,EFR} that Theorem \ref{Classification} cannot hold at any finite regularity. Still, the set of all solutions of the form \eqref{SEE2} forms a substantial set of such solutions and it is not difficult to imagine that solutions outside of that class only exist in highly degenerate situations. 
\subsection{Local and global structure of the set of steady states}
In light of the above, as a first step to studying the global properties of steady Euler solutions, it makes sense to study the global properties of solutions to all semilinear elliptic equations. While there is a considerable literature on semilinear elliptic equations, the author is unaware of a concerted effort to gain information about the global structure of the set of all solutions. In particular, a very interesting problem relates to the connectedness properties of this set. Let us give it a name in a particular setting: For $k\geq 2,$ we set
\[\mathcal{S}_k:=\{\psi\in C^k(\mathbb{T}^2): \{\psi,\Delta\psi\}=0\}/\sim,\] where we say that $\psi_1\sim\psi_2$ if there exist $\lambda\in(0,\infty),$ $c\in\mathbb{R},$ and $v\in\mathbb{R}^2$ so that 
\[\psi_2=\lambda\psi_1(x+v)+c.\]
It is clear that $\mathcal{S}_k$ is a closed subset of $C^k(\mathbb{T}^2).$ A first problem can be stated informally as follows. 
\begin{problem} \label{HotelProblem}
Does $\mathcal{S}_k$ have any nontrivial isolated points? Is $\mathcal{S}_k$ connected? If not, what are the dynamical implications of the multiple components? For $k$ large, can $\mathcal{S}_k$ be described globally as a suitably smooth Banach manifold away from a set of singular points of quantifiably lower dimension?
\end{problem}
Based on our quite limited knowledge, it is natural to guess that $\mathcal{S}_k$ will indeed be a nice infinite-dimensional manifold away from a few possibly degenerate points. A local result of this type was established in \cite{ShnirelmanBanach}. It is not clear, however, how bad it can behave near said degenerate points nor how many such points exist. In the case that $\psi_*$ is an analytic, Morse, steady Euler solution, it will satisfy \eqref{SEE2} for some $F=F_*\in C^\omega(\mathbb{R}).$ The local structure of $\mathcal{S}_k$ in a neighborhood of $\psi$ is then related to the invertibility properties of the Schr\"odinger operator
\[\mathcal{L}_*:=\Delta-F_*'(\psi_*)\] and the full linearized operator:
\[\mathcal{E}_*:=\{\psi_*,\Delta\cdot\}-\{\Delta\psi_*,\cdot\}=\{\psi_*,\mathcal{L}_*\cdot\}.\]
In the case that $\text{ker}(\mathcal{L}_*)=\{0\},$ we can give a very satisfactory description of the local structure of $\mathcal{S}_k$
\begin{theorem}
Assume $\psi_*$ is an analytic Morse function for which $\text{ker}(\mathcal{L}_*)=\{0\}.$ Then, if we fix $k\geq 3,$ there exists $\epsilon>0$ so that $\mathcal{S}_k\cap B_{\epsilon}(\psi_*)$ can be represented as a graph over $C^k\cap \text{ker}(\mathcal{E}_*).$
\end{theorem}
\noindent This can be seen as a manifestation of the implicit function theorem in infinite dimensions, with the essential aspects of the argument appearing in Section 2.6 of \cite{CZEW}. Analyticity of $\psi_*$ is probably not necessary. An example showing the necessity of the condition on the triviality kernel of $\mathcal{L}_*$ is also given in the same paper, where a "singular point" in the language of Problem \ref{HotelProblem} is studied. Such implicit function type theorems can also be found in the earlier works \cite{CS,CDG}. 

In view of the preceding Theorem, the main issue is to understand the local structure of steady states near points where the kernel of $\mathcal{L}_*$ is nontrivial. In general, this appears to be a very difficult problem and our lack of understanding of this case is the primary obstruction to resolving Problem \ref{HotelProblem}. 

\subsection*{Dynamics near Steady Solutions}
The dynamics near steady solutions to the Euler equation turns out to be extremely rich and diverse, depending on the particular steady solution we start with. In two dimensions, they can be used to establish long-time norm growth \cite{N,D,KS,Z}, nontrivial periodic and quasi periodic motion \cite{BHM,HHM,HHR,HR, HHR2,DDMW,DDMW2,GIP}, long-time relaxation \cite{BM}, and even nonuniqueness at low regularity \cite{DS,DSReview,Vishik,ABC}. In three dimensions they can be used to a similar effect and also to construct finite-time singularities \cite{EP,CMZIPM}, as we shall discuss later. For now, we will content ourselves with looking at two effects of norm growth that we learn from studying steady states, both of which will be discussed in more detail later: shearing and hyperbolic stretching.

\section{Asymptotic Behavior}
One of the most interesting aspects of the 2d Euler equation is the phenomenon of infinite-time relaxation. The long-time relaxation of 2d ideal fluids, despite the lack of a direct relaxation mechanism, is confirmed by numerous numerical experiments, statistical mechanics arguments, as well as personal experience. 
\subsection{Passive scalars and inviscid damping}
Let us begin with a simple explanation that comes from the study of transport by autonomous velocity fields. If $\Omega$ is a smooth bounded domain and $u:\Omega\rightarrow\mathbb{R}^2$ is tangent to $\partial\Omega$ and is divergence-free, it is well known that $\Omega$ is foliated by (i) periodic loops and (ii) trajectories that either start and end at zeros of $u$ (fixed points). For generic $u,$ this means that almost every trajectory is periodic. The period is also generically nonconstant. This leads to a concept called phase mixing:
\begin{theorem}\label{PassiveScalarThm}
Let $\Omega\subset\mathbb{R}^2$ be a smooth bounded domain and assume that $\psi\in C^2(\bar\Omega)$ is constant on the connected components of $\partial\Omega$. Let $u=\nabla^\perp\psi,$ and assume the associated Lagrangian flow map has a nowhere constant period function. Then, all $L^2$ solutions to  
\begin{equation}\label{PassiveScalar}\partial_t f+u\cdot\nabla f=0\end{equation}
satisfy
\[u\cdot\nabla f\rightharpoonup 0\] as $t\rightarrow\infty.$ In other words, every $L^2$ solution to \eqref{PassiveScalar} converges weakly to its $L^2$ projection onto the linear space of steady solutions of \eqref{PassiveScalar}.
\end{theorem}
\begin{remark}
In general, the problem \eqref{PassiveScalar} can accommodate time-periodic and quasi-periodic solutions. These can also be embedded directly into the Euler equation in any domain using the idea of \cite{CrF}. There are also a number of constructions of periodic and quasi-periodic solutions in the Euler equation \cite{BHM} and related models \cite{GIP}. 
\end{remark}
There are a number of ways to prove this theorem. Quantitative versions can be found in \cite{BCE} along with a recent work \cite{DEG2}. As is explained in \cite{DEG2}, in non-degenerate settings\footnote{When $\psi$ is a Morse function and the period function is also Morse.}, it is possible to construct time-independent vector fields $a$ and $b$ for which: \[a\cdot\nabla,\qquad b\cdot\nabla+t\, a\cdot\nabla \] commute with the transport operator $\partial_t+u\cdot\nabla.$ As a result, it is not difficult to see that any solution to \eqref{PassiveScalar} satisfies:
\[|(a\cdot\nabla f,\phi)_{L^2}|\leq \frac{C}{t}|f|_{H^1}|\phi|_{H^1}, \] as $t\rightarrow\infty.$ 

Since the 2d Euler equation is simply a nonlinear transport equation for the vorticity:
\[\partial_t\omega+u\cdot\nabla\omega=0,\]
\[u=\nabla^\perp\Delta^{-1}\omega,\] it is natural to ask whether phase mixing can occur in the 2d Euler equation. That it occurs in the linearized Euler equation, say around the solution $u_*(x,y)=\begin{pmatrix}y\\
0\end{pmatrix},$ was observed in the 19th century. That this also occurs in some settings for the full Euler equation was proven in a landmark paper \cite{BM}. See also \cite{IJ,MZ} for some recent developments. Let us give a \emph{consequence} of the main theorem of \cite{BM}, which is too detailed to discuss here.  
\begin{theorem}\label{BM}
Fix $u_*(x,y)=\begin{pmatrix}y\\
0\end{pmatrix},$ on $\mathbb{T}\times\mathbb{R}.$ Then, every solution to the Euler equation that is initially close to $u_*$ in a suitable topology satisfies:
\[u\rightarrow \begin{pmatrix}u_\infty(y)\\
0
\end{pmatrix}\] strongly in $L^2$ as $t\rightarrow\infty.$
\end{theorem}
The main idea is to build off of the decay mechanism in the linear problem:
\[\partial_t\omega+y\partial_x\omega=0\]
\[u=\nabla^\perp\Delta^{-1}\omega.\] Using the proof of Theorem \ref{PassiveScalar}, it is not difficult to check that the linearized Euler equation around $(y,0)$ satisfies:
\[|u_1|_{L^2}\leq \frac{C}{t}|\omega(0,\cdot)|_{H^1},\qquad |u_2|_{L^2}\leq \frac{C}{t^2}|\omega(0,.)|_{H^2}.\]
A major difficulty in the proof of Theorem \ref{BM} is that the somewhat slow decay of $u$ in the linear problem is accompanied by the growth of derivatives of $\omega,$ while the nonlinear term contains derivatives of $\omega.$ In the end, the asymptotic stability theorem is proven in certain infinite regularity spaces and, to the author's knowledge, it is not known what happens for small Sobolev perturbations (except at low regularity, where other time-periodic behavior can occur \cite{LZ,DN}).

\subsection{Emergence of special solutions}
Theorems \ref{PassiveScalarThm} and \ref{BM} suggest that the long-time behavior of 2d Euler solutions could be much simpler than what one would first guess. To make precise the conjectured behavior, let us fix a smooth bounded domain $\Omega\subset\mathbb{R}^2$ and restrict ourselves to the class of bounded vorticity solutions on $\Omega.$ We also let $S_t:L^\infty(\Omega)\rightarrow L^\infty(\Omega)$ denote the 2d Euler solution operator (acting on the vorticity).  
\begin{definition}
Define \[\mathcal{A}_{\infty}:=\{f\in L^\infty: \,S_{t_n}(\omega_0)\rightharpoonup f\,\,\, \text{for some} \,\,\omega_0\in L^\infty\,\,\, \text{and}\,\,\, t_n\rightarrow\infty\}.\] This is the weak limit set of all $L^\infty$ solutions to the 2d Euler equation on $\Omega.$
\end{definition}
\noindent We can now state what can be called the Relaxation Conjecture, due to Shnirelman \cite{Shnirelman2} and \v{S}ver\'ak \cite{SverakNotes}.
\begin{conjecture}\label{RelaxationConjecture} We have that
\begin{itemize}
\item $\mathcal{A}_\infty$ consists of those $f\in L^\infty $ for which $\{\mathcal{S}_t(f)\}_{t\in\mathbb{R}}$ is pre-compact in $L^2.$ 
\item For generic $f\in L^\infty$, we have that $\{S_{t}(f)\}_{t\in\mathbb{R}}$ is not pre-compact in $L^2.$  
\end{itemize}
\end{conjecture}
\begin{remark}
Under the assumptions of Theorem \ref{PassiveScalarThm}, the analogue of $\mathcal{A}_\infty$ consists just of steady solutions to \eqref{PassiveScalar}. For the general passive scalar case, $\mathcal{A}_\infty$ consists of a wider class that can also include time-periodic and quasi-periodic solutions. Moreover, the analogue of the above Conjecture holds for the passive scalar setting merely under the assumption that the velocity field is $C^1$ (and autonomous, of course). See Figure \ref{LongTimeCaricature} for a caricature of the dynamics. 
\end{remark}
\begin{figure}
  \centering\includegraphics[width=0.6\linewidth]{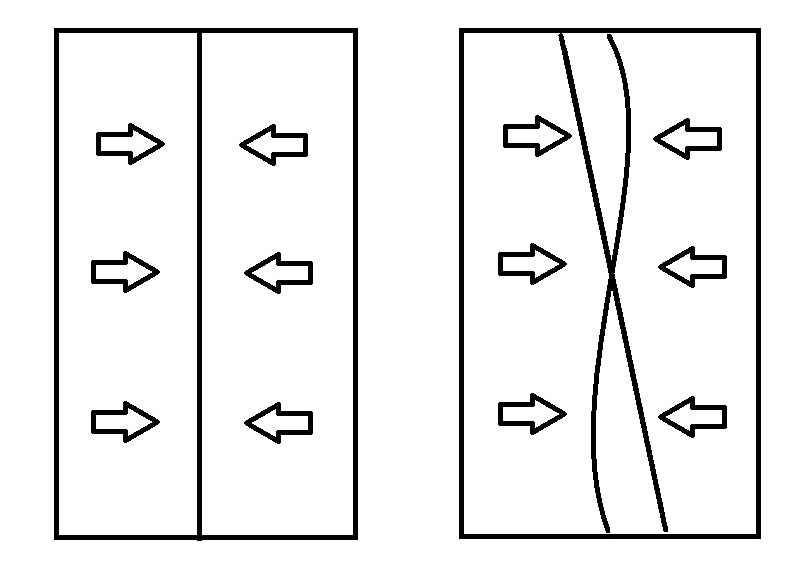}
  \caption{Caricature of the long-time dynamics for passive scalars (left) and the 2d Euler equation (right). The set in the middle, in each case, is the set $\mathcal{A}_\infty$ that attracts all solutions. In the passive scalar case it is an infinite dimensional linear space that consists of steady states, time-periodic solutions, etc, though generically it only consists of steady states. For the 2d Euler equation, it is conjectured to consist of those solutions lying on compact orbits (including steady states, periodic solutions, etc.). While in the passive scalar case the limit set is linear, the limit set in the 2d Euler case is certainly nonlinear and may be singular, as discussed above.}
  \label{LongTimeCaricature}
\end{figure}

This conjecture has only been established in small neighborhoods of shear flows in very high regularity \cite{BM,MZ,IJ}. An analogue has also been established for scale-invariant solutions in \cite{EMS}.

\subsection{Lyapunov stability}
Stepping back from asymptotic stability and mixing for a moment, it is natural to think that solutions to the Euler equation should enjoy certain stability properties. Indeed, if we view an Euler solution as the optimal way to accomplish a certain task, it is natural to think that solutions to the Euler equation should enjoy some kind of stability properties. This has been rigorously established for a wide variety of steady flows using a variational principle that goes back to work of Lord Kelvin \cite{Kelvin} and also Arnold \cite{ArnoldKhesin}. An example of such a stability theorem is
\begin{theorem}
Let $\Omega\subset\mathbb{R}^2$ be a smooth bounded domain and assume that $\psi_*\in C^2(\bar\Omega)$ be a steady state solution to the Euler equation satisfying
\[\Delta\psi_*=F_*(\psi_*),\] for $F_*\in C^1.$ If $F_*'>0$ then $\psi_*$ is Lyapunov stable in $H^2.$ \end{theorem}
\noindent The proof of the theorem is based on observing that $\psi_*$ minimizes a certain conserved quantity built out of the energy and $\int F(\omega)$ for some properly chosen $F.$ Note that there are many more Lyapunov stable steady states, notably the case $F_*\equiv C.$ While this is a very classical subject and much is known, there remain several open problems about Lyapunov stability. One interesting problem relates to whether orbitally stable steady states are actually Lyapunov stable (see \cite{CG,WZ,EMaximizers,DM}). Another interesting problem is whether there is Lyapunov stability in the Yudovich class \cite{Yudovich}. Note that Lyapunov stability in $H^2$ implies Lyapunov stability in $W^{2,p}$ for $2\leq p<\infty,$ using the conservation laws. 

\subsection*{Dynamical consequences of Lyapunov stability}
Let us now take a brief look at the dynamical implications of Lyapunov stability. 
We fix a smooth region $\Omega$ that could either be a bounded subset of $\mathbb{R}^2$ or the periodic channel $\mathbb{T}\times[0,1].$ Let us also fix a Lyapunov stable $\psi_*$ with velocity $u_*$ and vorticity $\omega_*.$ 
In particular, we will want to understand what can be said about the flow map:
\[\frac{d}{dt}\Phi(x,t)=u(\Phi(x,t),t),\] when $u$ solves the Euler equation and is close to $u_*.$
A first observation that one can make is that the $L^2$ stability at the level of $\omega,$ which is transported, strongly constrains the flow of particles. In particular, since $\omega$ must remain close to $\omega_*$, and since $\omega=\omega_0\circ\Phi_t^{-1}$ and $\omega_0$ is close to $\omega_*,$ this puts a constraint on the flow map $\Phi_t.$ 
In particular, it is easy to see that  
\[|\omega-\omega_*|_{L^2}<\epsilon\implies |\omega_*\circ\Phi(\cdot,t)-\omega_*|_{L^2}<2\epsilon.\]
If $\omega_*$ is monotone, in some sense, this essentially puts a constraint on one component of $\Phi_t.$ 
In particular, Eulerian Lyapunov stability gives Lagrangian stability for the component of the flow that is orthogonal to the direction of transport.

One can also generalize this question beyond the case of "monotone" solutions to the Euler equation to general vector fields:
\begin{question}\label{Highway}
Assume that 
\[|u(\cdot,t)-u_*(\cdot)|_{H^1}<\epsilon\] for all $t\in\mathbb{R}$, with $\epsilon>0$ as small as we like. What can be said about the flow of particles generated by $u?$ As a particular case, if we take $u_*(x,y)=\begin{pmatrix}y\\ 0\end{pmatrix}$ on $\mathbb{T}\times [0,1],$ does such a $u$ necessarily exhibit a type of phase mixing? More concretely, could it be that all $L^2$ solutions to the transport equation with such a $u$ have compact orbits in $L^2?$
\end{question}
The question essentially asks whether "phase mixing" is stable in any sense--the hope being that the answer to the final question is \emph{negative.}  
A partial answer to this question was given in the paper \cite{DEJ}, where it was shown that "twisting" is stable. Let us define this precisely in a special case. Fix $\Omega=\mathbb{T}\times[0,1],$ the periodic channel, which can also be identified with an annulus. For $x\in\Omega,$ we may consider the winding number, $N(x,t)$ of $\Phi_t(x)$ around the center of the annulus.
\begin{definition}
    We say that $\Phi_t$ is twisting if the winding number $N(\cdot,t)$ satisfies:
    \[\sup_{x,y\in \Omega}|N(x,t)-N(y,t)|\rightarrow\infty\] as $t\rightarrow\infty.$ We say that a vector field $u$ is twisting if its corresponding flow map $\Phi_t$ is twisting.    
\end{definition}
A first result proved in \cite{DEJ} is
\begin{theorem}
If $u_*$ is autonomous and twisting and $|u(\cdot,t)-u_*(\cdot)|_{L^2}$ is sufficiently small for all $t>0,$ then $u$ is twisting.
\end{theorem}
In fact, much more is shown. In the above setting, it is shown that 
\begin{equation}
\label{ArnoldDiff}|u_*(\Phi_2)-u_*(y)|_{L^2}\leq C\sqrt\epsilon \log (1+t),
\end{equation} for all $t\geq 0.$ 
In particular, it would take exponentially long amount time for an order 1 region of particles to move vertically. This is reminiscent of the Nekhoroshev bounds on Arnold diffusion \cite{Nekh}. Similarly, it is shown that the lift $\tilde\Phi(\cdot,t)$ of $\Phi(\cdot,t)$ to the universal cover of $\Omega$, $\mathbb{R}\times [0,1],$ satisfies:
\begin{equation}\label{HighwayStability}|\tilde\Phi_1-t u_*(\tilde\Phi_2)-x_1|_{L^2}\leq C\sqrt{\epsilon} |t|,\end{equation} for all $t.$
\eqref{HighwayStability} is an example of Lagrangian stability in the sense that it says that the new flow $\Phi_t,$ of $u$, when pulled back by the old flow $\Phi_t^*$ of $u_*$, does not grow too fast. 

Let us note that all of these results in fact only used the closeness of $u$ to $u_*$ in $L^2$. It would be very interesting to see what better results could be obtained using the full assumption of closeness in $H^1$ in Question \ref{Highway}. One particularly interesting question is whether it is possible to remove the $\log$ in \eqref{ArnoldDiff}. Doing so would have a few nice consequences. Regarding phase mixing, it is clear that just using closeness in $L^2,$ it is easy to perturb $u_*$ so that it becomes piecewise constant and thus exhibits no phase mixing. It seems difficult to achieve this with a small $H^1$ perturbation. Showing that phase mixing is stable in $H^1$ or even stronger norms, i.e. giving a \emph{negative} answer to the final part of Question \ref{Highway}, could represent a major step towards establishing the second part of the Relaxation Conjecture \ref{RelaxationConjecture} for 2d Euler solutions.


\section{Local and Global Solvability}

We now transition to discuss local and global solvability. The questions of local and global solvability for the Euler equation have been studied rigorously since the 1920's. A first question one can ask is: 
\begin{question}
In which function spaces $X$ should we search for solutions to the Euler equation?
\end{question}
Looking at the equation:
\begin{equation}\label{Euler1}\partial_t u+u\cdot\nabla u+\nabla p=0,\end{equation}
\begin{equation}\label{Euler2}\text{div}(u)=0,\end{equation} it makes sense to seek solutions that are continuously differentiable in $x$ and $t$. However, it is a fact of life, as L. Nirenberg once remarked, that the integer spaces $C^k$, $k\in\mathbb{N},$ are not generally good spaces for solvability. This was later confirmed for the Euler equation in \cite{EM,BL}. Despite this, it makes sense to term any solution $u$ that happens to be in the class $C^{1}_{x,t}$ a \emph{classical} solution\footnote{In order to avoid pathologies at spatial infinity, when we say classical solution on $\mathbb{R}^3,$ they should satisfy some mild conditions as $|x|\rightarrow\infty$ or satisfy some symmetry condition. The essential issue is to be able to define the pressure in a unique way \cite{EJSymm}. Otherwise, it is possible to take $u(x,t)=\begin{pmatrix}f(t)\\
0\end{pmatrix}$ and $p(x,t)=-f'(t)x_1,$ with $f(t)$ arbitrary.}, since it satisfies the equation in the classical sense. To the author's knowledge, the first local existence result appeared in separate works of Lichtenstein \cite{Lichtenstein} and Gunther \cite{Gunther} in spaces just a little stronger than $C^1.$ Let us recall the $C^{1,\alpha}$ spaces.
 For $0\leq \alpha\leq 1$ and any open $\Omega\subset\mathbb{R}^d$, we say that $f\in C^{1,\alpha}(\Omega)$ if $f\in C^1(\Omega)$ and
\[|f|_{C^{1,\alpha}}:=\sup_{x\in\Omega} |\nabla f(x)|+\sup_{x\not=y} \frac{|\nabla f(x)-\nabla f(y)|}{|x-y|^\alpha}<\infty.\]

\begin{theorem}\label{Classical}
$\Omega$ be a smooth bounded domain in $\mathbb{R}^d.$ For any $0<\alpha<1$ and $u_0\in C^{1,\alpha},$ tangent to $\partial\Omega$, there exists a $T_*>0$ and a unique solution $u\in C^{1,\alpha}(\bar\Omega\times[0,T_*))$ to \eqref{Euler1}-\eqref{Euler2} that is tangent to $\partial\Omega$ and with $u|_{t=0}=u_0$.
\end{theorem} It was proven shortly thereafter \cite{Holder,Wolibner} that such solutions are global when $d=2.$
\begin{theorem}\label{global}
When $d=2$, the $C^{1,\alpha}$ solutions of the preceding theorem can be extended for all time. 
\end{theorem}
\begin{remark}
As far as the author knows, outside perhaps the case $d=2,$ it is not known what is the largest space embedded in $C^1$ where one can give a good local theory for solutions. Here, we are asking specifically what the correct condition on the modulus of continuity of $\nabla u$ or $\omega$ is. When $d=2,$ this problem has been studied in \cite{Koch,KSK} where the Dini condition naturally appears. This does not appear to be the right condition when $d\geq 3.$ Outside of this, there is quite a bit known about the solvability of the Euler equation in a wide variety of spaces.
\end{remark}
A notoriously difficult open problem is to determine whether solutions can be extended globally when $d\geq 3.$ In many cases, this problem is still open. However, as far as the classical solutions of Theorem \ref{Classical} go, this is resolved by
\begin{theorem}\label{SingularityClassical}
There exists a classical solution $u\in C^{1,\alpha}(\mathbb{R}^3\times[0,1))$ with $\alpha>0$ so that 
\[\lim_{t\rightarrow 1^-}|\nabla u(t)|_{L^\infty}=+\infty.\] In particular, $u$ cannot be extended to be a classical solution on $\mathbb{R}^3\times[0,1]$.\end{theorem}
\begin{remark}
In the above theorem, $\alpha>0$ is small and is used as a parameter to close a certain nonlinear fixed point argument, as we shall discuss later. In the symmetry class considered in the first proof \cite{E_Classical} of Theorem \ref{SingularityClassical}, the smallness of $\alpha$ was necessary. Later constructions \cite{ChenHou1,EP} were done where this restriction may not be necessary.
\end{remark}

\section{Singularity Formation}
As we have mentioned in the preceding section, singularities do form for classical solutions to the Euler equation. The purpose of this section is to discuss the mechanisms for singularity formation and how it is proven to occur in some cases.  

\subsection{Necessary conditions}
We should first discuss what has to happen for solutions to become singular. There are two conceptually important results, both of which are inspired by the global regularity in 2d and the dynamics of vorticity. Before we mention them, let us thus recall the vorticity equation:
\begin{equation}\label{VorticityTransport}\partial_t\omega+[u,\omega]=0,\end{equation}
\begin{equation}\label{BSLaw}u=\nabla\times (-\Delta)^{-1}\omega.\end{equation}
The first necessary condition for a singularity is the celebrated Beale-Kato-Majda criterion, which is what distinguishes incompressible singularities:
\begin{theorem}\label{BKM}
If the solution from Theorem \ref{Classical} cannot be continued past some time $T_*<\infty,$ then the vorticity must become unbounded as $t\rightarrow T_*.$ In fact, it must be that
\[\lim_{t\rightarrow T_*^-}\int_{0}^t|\omega|_{L^\infty}=+\infty.\] 
\end{theorem}

While this result is now considered elementary, there are several important conceptual and practical applications of Theorem \ref{BKM}. Since the vorticity is Lie transported by $u$, the first is that it allows us to focus our attention on a simpler question: \begin{question}\label{Question} If a vector field $\omega$ is Lie transported by another vector field $u,$ what qualitative features of $u$ and $\omega_0$ lead to the pointwise growth of $\omega$? Are these qualitative features consistent with \eqref{BSLaw}?\end{question}
\noindent The first question is asked independently of the Euler equation in that we need not assume any relation between $u$ and $\omega.$ It would be very good to understand the first question even in the case where $u$ is time-independent. Understanding some aspect of these questions will be the key to the main results we will discuss. One practical use of Theorem \eqref{BKM} is in numerical simulations. The system \eqref{VorticityTransport}-\eqref{BSLaw} is closed, so the pointwise growth of $\omega$ can be checked numerically much more simply than, say, the growth of $\nabla u.$ Let us remark, finally, that growth mechanisms for solutions to the Euler equation coming directly from the velocity formulation are scarce. 

A second result, established by Constantin, Fefferman and Majda \cite{CFM96} after previous work by Constantin and Fefferman \cite{CF} on the Navier-Stokes system, gives a geometric constraint on the growth of vorticity. 

\begin{theorem}\label{CFM}
If the solution from Theorem \ref{Classical} cannot be continued past some time $T_*<\infty$ and if the velocity field $u$ is uniformly bounded up to $T_*,$ then the direction of the vorticity vector must become irregular as $t\rightarrow T_*.$ In fact, if $\xi=\frac{\omega}{|\omega|}$ it must be that
\[\lim_{t\rightarrow T_*}\int_0^t |\nabla \xi|_{L^\infty}\rightarrow\infty.\]
\end{theorem}
\noindent This can be seen as a generalization of Theorem \ref{global}. An important use of Theorem \ref{CFM} is for checking the reliability of numerically produced singularity scenarios. One important concrete idea we can gain from this is that a singularity might occur more readily at points where the vorticity vanishes (like at an axis of symmetry). This is because the direction of vorticity may naturally be discontinuous at places where the vorticity vanishes, even if the solution is smooth. Several of the most important numerical simulations, particularly the one done in the landmark work of Luo and Hou \cite{LuoHou}, used these blow-up criteria as ways to check for true singularity formation.

\subsection*{Examples with infinite vorticity growth}
Let us now observe that vorticity growth in 3d is at least as fast as the maximal gradient growth in 2d. In particular, since gradient growth is prevalent in 2d Euler solutions \cite{DEJ,DEReview}, vorticity growth is likely prevalent in 3d. 
\begin{proposition}
Let $\Phi_t(x)$ be the flow map of a smooth 2d Euler solution on some two-dimensional domain $\Omega,$ either a smooth bounded subset of $\mathbb{R}^2$ or $\mathbb{T}^2.$ Let $\lambda(t)=|\nabla\Phi_t|_{L^\infty}.$ Then, there exists a smooth solution to the 3d Euler equation on $\Omega\times\mathbb{T}$ for which:
\[|\omega(t)|_{L^\infty}\geq \lambda(t),\]
for all $t\geq 0.$
\end{proposition}
\begin{proof}
First observe that $|\nabla\Phi_t|_{L^\infty}=|\nabla(\Phi_t^{-1})|_{L^\infty},$ by incompressibility. Next, there must exist a smooth $g_0$ for which $|\nabla(g_0\circ\Phi_t^{-1})|\geq c\lambda(t).$ Indeed, in the case that $\Omega$ is a bounded subdomain of $\mathbb{R}^2,$ then 
\[\sum_{i=1}^2 |g_0^i\circ\Phi_{t}^{-1}|_{C^1}\geq \lambda(t),\] for $g_0^i=x_i.$ In the case where $\Omega=\mathbb{T}^2,$ we can choose finitely many $g_0^i$ so that $g_0^i$ is linear on subdomains of $\mathbb{T}^2.$  This is easily achievable with eight such $g_0^i.$ Finally, observe that any solution to the 2d Euler equation can be extended to a solution of the 3d Euler equation with a third component that is just transported by the 2d velocity\footnote{These are called $2\frac{1}{2}$ dimensional solutions, though this name is somewhat misleading.}. 
\end{proof}

As a consequence of the landmark result \cite{KS}, we see that the vorticity in 3d can grow double exponentially fast pointwise as $t\rightarrow\infty$ in the periodic cylinder.  If we believe that genuine three-dimensional solutions exhibit even a little more roughness, it should not be surprising that a singularity can form in finite time. Note that while the result of \cite{KS} uses the boundary in a crucial way, it is not difficult to get exponential growth of the vorticity in $\mathbb{T}^3$. Examples of algebraic growth were given in \cite{BardosTiti}, for example.

\subsection{Model equations and non-descriptive arguments}
Before jumping into one approach to singularity formation in the Euler equation, via the analysis of self-similar singularities, let us take a moment to look at model equations and other potential arguments for singularity formation. While such techniques do not generally give us a very detailed and precise description of the singularities, they can reveal deep structures in the equations that tell us \emph{why} singularities form. One of the first models that appeared in the study of singularity formation is the Constantin-Lax-Majda model  \cite{CLM}:
\[\partial_t\omega=\omega H(\omega),\]
\[\partial_x u=H(\omega),\] where $H$ is the Hilbert transform. This was designed as a toy model to understand the effects of vorticity stretching in a non-local model. It actually turns out to be kind of embedded in the Euler equation, in a certain limit. Singularity formation for this model was shown in the original paper \cite{CLM} by finding an exact solution formula for solutions. It turns out that this solution formula shows that singularities form for a large class of data and that the singularities are self-similar. Let us mention also that it was shown in \cite{BE} that \emph{any} scalar equation of the type 
\[\partial_tf=fR(f),\] with $R$ a linear operator satisfying some mild assumptions, exhibits singularity formation. It is an interesting problem to see what happens in the presence of transport terms.  
 One such case is given in the De Gregorio model \cite{DG1, DG2}:
\[\partial_t \omega + u\partial_x\omega=\omega\partial_x u,\]
\[\partial_x u=H(u).\] This can be seen as the Constantin-Lax-Majda model with transport. Singularity formation and global regularity for this model has been shown in a few contexts \cite{EJDG, JSS, Chen}. It is expected that there is global regularity for smooth solutions to this model when posed on $\mathbb{S}^1$. A great deal of evidence for this was given in the work \cite{Chen}. See also \cite{JSS}.

As we move to two dimensions, two important models of the axi-symmetric 3d Euler equation are the Boussinesq and IPM systems. The IPM system is simply:
\[u+\nabla p= -\rho e_2\]
\[\partial_t\rho+u\cdot\nabla\rho=0.\] $\rho$ can be seen as a density. When $\rho$ is independent of $x_1,$ there is no motion $(u\equiv 0).$ When $\rho$ depends on $x_1,$ the fluid moves due to the effect of gravity. Indeed,  when $\partial_1\rho(x_*)>0,$ at some point $x_*$, particles to the right of $x_*$ will locally tend downwards (since they are more heavy). This can be seen mathematically by observing that curl of $u$ is simply $-\partial_1\rho.$ A natural setting for singularity formation in this setting is in an initial configuration with "heavy over light" fluid. See \cite{ZlatosIPM,CMZIPM} for recent outstanding progress on singularity formation in the IPM equation utilizing this idea. Let us mention, about the paper \cite{ZlatosIPM}, that the main mechanism for singularity formation is that solutions to the equation tend to want to stratify. The equation's desire to stratify is so strong that it could potentially occur in finite time.  Applying such an idea to smooth solutions or to the Boussinesq and Euler systems would be quite significant. The setting of \cite{ZlatosIPM} is indeed quite similar to the singularity reported by Luo and Hou \cite{LuoHou}. A similar physical mechanism for singularity formation was used in \cite{EJB, EJ3dE, EJOct} along with scale-invariance to establish singularity formation for strong solutions in domains with corners and on $\mathbb{R}^3.$ Let us also remark finally that there have been various attempts at getting a singularity for the Euler equation and related models by studying models on the Fourier side \cite{KatzPavlovic,Tao,Miller}. Let us also point out that there was also important progress on singularity formation in \cite{CMZZ,CMZ3DE} using what the authors call "vortex-layer cascades." 

\subsection{Graphical representations of various blow-up scenarios}
Below we give graphical descriptions of a few scenarios for singularity formation in the Boussinesq (and IPM) system that have direct implications for the axi-symmetric 3d Euler system and the axi-symmetric Euler equation without swirl. In each scenario, we give a caricature of an initial profile (left) along with its evolution after some time (right).  

\begin{figure}
  \centering\includegraphics[width=0.35\linewidth]{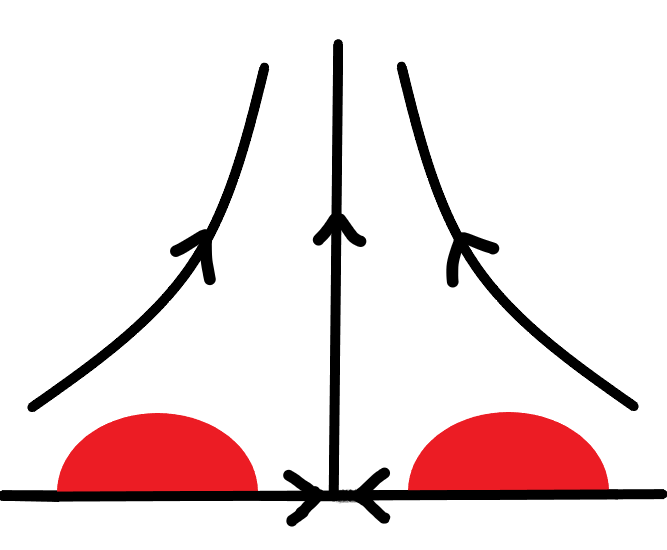} \qquad \includegraphics[width=0.35\linewidth]{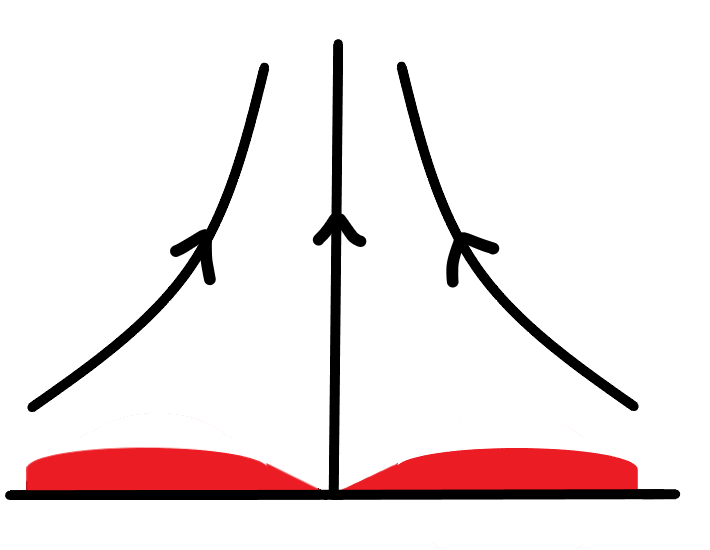}
  \caption{Caricature of the singularity reported by Luo and Hou in \cite{LuoHou}. The horizontal axis, in the picture, is a solid boundary. The red regions should be seen as regions of fluid with high density outside of which is fluid with low density. Since gravity is pointing down, the light fluid in the center is ejected upwards and the high density bumps try to settle. The singularity reported in \cite{LuoHou} predicts that a type of collision occurs at the origin in the picture (as on the right). The intensity of the "collision" at the origin should depend on the degree of vanishing of flatness of the density at $x=0$, as in singularities in the Burgers equation \cite{CGM}. This is also the geometric scenario considered in \cite{ZlatosIPM}, with a corner in \cite{EJB,EJ3dE}, and the numerical work \cite{Wang}. See also \cite{CKY1,CKY2,CHH} for works on models of this scenario. }
  \label{HouLuoCaricature}
\end{figure}

\begin{figure}
\begin{center}
\includegraphics[width=0.3\linewidth]{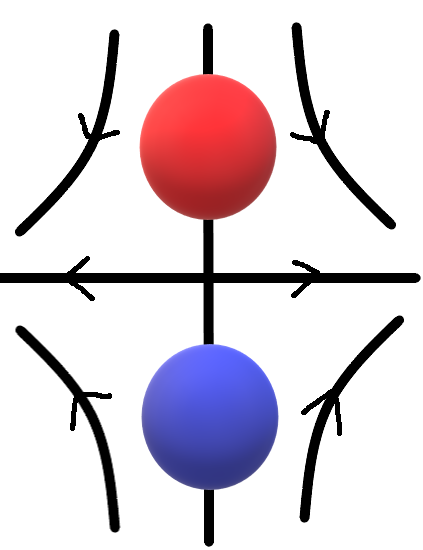}
\qquad 
\includegraphics[width=0.3\linewidth]{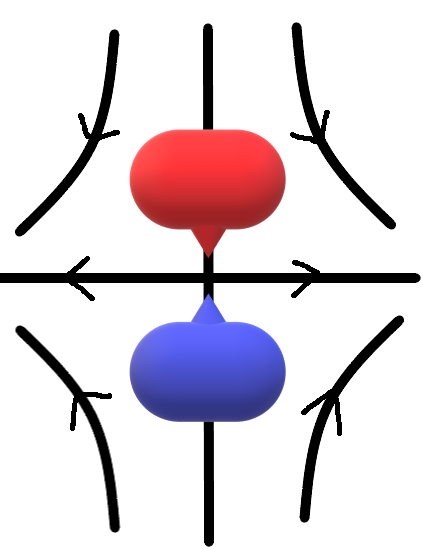}

\end{center}
\caption{Caricature of the singularity constructed in \cite{E_Classical}. A caricature of the data is on the left, while the solution just prior to the singularity is on the right.  Particles flow down and out symmetrically along each axial plane. The initial vorticity is axi-symmetric without swirl and vanishes only very weakly on the symmetry axis. This allows for particles at the axis to come down much faster than those away from the axis. At the final time, the vorticity becomes unbounded at the origin and the velocity develops a cusp discontinuity at $x=0$ while remaining smooth away from the symmetry axis.}
\end{figure}

\begin{figure}
\begin{center}
	\begin{tikzpicture}[domain=-10:10, scale=0.28]
		\draw[->] (-10,0) -- (10,0);
		\draw[->] (0,-10) -- (0,10) ;
		
		\draw[color=red,  dotted, very thick]   plot (\x,{2});
		\draw[color=red, dotted, very thick]   plot (\x,{4});
		\draw[color=red, dotted, very thick]   plot (\x,{6});
		\draw[color=red, dotted, very thick]   plot (\x,{8});
		
		\draw[color=blue, dotted, very thick]   plot (\x,{-2});
		\draw[color=blue, dotted, very thick]   plot (\x,{-4});
		\draw[color=blue, dotted, very thick]   plot (\x,{-6});
		\draw[color=blue, dotted, very thick]   plot (\x,{-8});
		
		\usetikzlibrary {arrows.meta} 
		\draw[color=black, domain=1:10, ->, arrows = {-Stealth[scale=2]}]   plot (\x,{10/\x});
		\draw[color=black, domain=1:10, ->, arrows = {-Stealth[scale=2]}]   plot (\x,{-10/\x});
		\draw[color=black, domain=-1:-10, <-, arrows = {-Stealth[scale=2]}]   plot (\x,{10/\x});
		\draw[color=black, domain=-1:-10, ->, arrows = {-Stealth[scale=2]}]   plot (\x,{-10/\x});
		
		\node[draw, circle, scale=2] at (5,5) {$+$};
		\node[draw, circle, scale=2] at (-5,-5) {$+$};
		\node[draw, circle, scale=2] at (5,-5) {$-$};
		\node[draw, circle, scale=2] at (-5,5) {$-$};

	\end{tikzpicture}
\qquad
	\begin{tikzpicture}[domain=-10:10, scale=0.28]
		\draw[->] (-10,0) -- (10,0);
		\draw[->] (0,-10) -- (0,10) ;
		
		\draw[color=red,  dotted, very thick]   plot (\x,{4 - 3.5*pow(1.4, -pow(0.75*\x, 2))});
			\draw[color=red,  dotted, very thick]   plot (\x,{6 - 3.5*pow(1.4, -pow(0.75*\x, 2))});
			\draw[color=red,  dotted, very thick]   plot (\x,{8 - 3.5*pow(1.4, -pow(0.75*\x, 2))});
			\draw[color=red,  dotted, very thick]   plot (\x,{10 - 3.5*pow(1.4, -pow(0.75*\x, 2))});
		
		\draw[color=blue,  dotted, very thick]   plot (\x,{-4 + 3.5*pow(1.4, -pow(0.75*\x, 2))});
		\draw[color=blue,  dotted, very thick]   plot (\x,{-6 + 3.5*pow(1.4, -pow(0.75*\x, 2))});
		\draw[color=blue,  dotted, very thick]   plot (\x,{-8 + 3.5*pow(1.4, -pow(0.75*\x, 2))});
		\draw[color=blue,  dotted, very thick]   plot (\x,{-10 + 3.5*pow(1.4, -pow(0.75*\x, 2))});
		
		\usetikzlibrary {arrows.meta} 
		
		\node[draw, circle, scale=1] at (7,2) {$+$};
		\node[draw, circle, scale=1] at (-7,-2) {$+$};
		\node[draw, circle, scale=1] at (7,-2) {$-$};
		\node[draw, circle, scale=1] at (-7,2) {$-$};
		
		\node[draw, circle, fill, scale=1/2] at (8,8) {};
		
		\draw[->, very thick] (8, 8) -- (8, 10);
		
		\node[draw, circle, fill, scale=1/2] at (1.8,8) {};
		
		\draw[->, very thick] (1.8, 8) -- (1.8, 6);
		
		\node[draw, circle, scale=1, color=black, very thick] at (5,8.2) {$+$};
	\end{tikzpicture}
\end{center}
\label{UnstableSteadyScenario}
\caption{Caricature of the singularity constructed in \cite{EP}. The doted lines indicate level sets of the density while the circular $+$ and $-$ figures represent vorticity of positive and negative sign, respectively. The left figure represents the initial configuration with unstably stratified density plus a small perturbing vorticity. The right figure represents the short time evolution of the level sets of density by the perturbing vorticity. The change in the shape of the level sets of the density leads to the creation of \emph{more} vorticity, which is represented by the $+$ sign in the upper right corner. The main barrier to singularity formation in this scenario is the flattening of the level sets of the density due to incompressibility due to stretching in the horizontal direction. The singularity constructed in \cite{EP} overcomes this issue in the setting of $C^{1,\alpha}$ solutions that are smooth in the angular direction at $x=0.$ It may very well be that truly smooth solutions develop a singularity in this setting.  }
\end{figure}
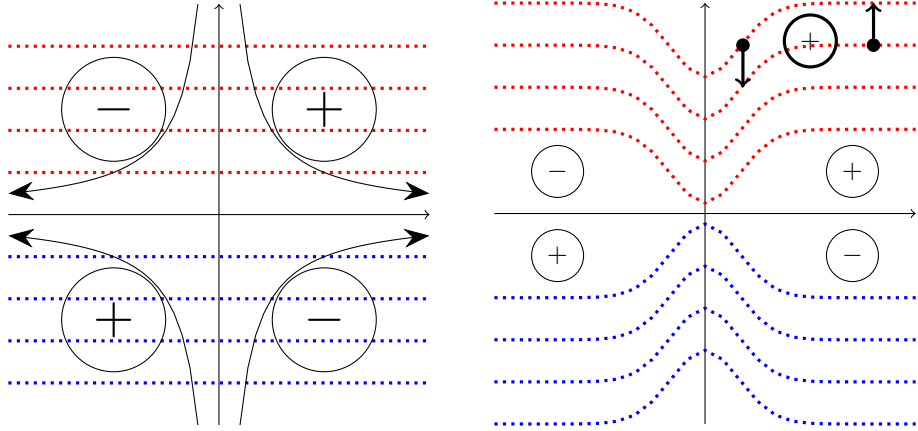

\subsection{Self Similar Singularities}
One of the most serious difficulties with proving singularity formation rigorously is the inherent instability associated with growth. This, coupled with the non-locality of the Euler equation, makes it easy to imagine the many things that could go wrong when trying to actually prove singularity formation. For example, the growth at some point could trigger stronger growth at some nearby points that eventually depletes the original growth mechanism. For this reason, it would be desirable to search for singularities that are as simple and as tame as possible. A natural option is to search for self-similar singularities, which are just given by the rescaling of a single profile:
\[\omega(x,t)=\frac{1}{T-t} \Omega (\frac{x}{(T-t)^{\lambda}}),\] for some $\lambda\in\mathbb{R}$. The problem is then reduced to finding the pair $(\Omega,\lambda)$ which solve a \emph{stationary} problem:
\begin{equation}\label{SS1}\Omega+\lambda x\cdot\nabla \Omega + U\cdot\nabla\Omega=\Omega\cdot\nabla U,\end{equation}
\begin{equation}\label{SS2} U=\nabla\times (-\Delta)^{-1}\Omega.\end{equation}
There is an inherent difficulty in solving \eqref{SS1}-\eqref{SS2} that comes from non-locality. For local problems, such as in the singularity problem for compressible fluids \cite{CGM,MRRS}, \eqref{SS1} could be interpreted as an ODE with $x\cdot\nabla=r\partial_r$ giving us an evolution in $r$, though with a singular point at $r=0$ (and possibly other singular points, depending on the problem). The difficulty here is that \eqref{SS2} means that $U(x)$ is determined by $\Omega$ everywhere. In particular, if we view \eqref{SS1} as an ODE, we would have to see it as an ODE where the future and the past both influence the present in a nontrivial way. This presents a unique difficulty that has not yet been resolved satisfactorily. As a result, the main tool we currently possess to solve \eqref{SS1}-\eqref{SS2} is perturbative. This argument was used to prove Theorem \ref{SingularityClassical} in different contexts in \cite{E_Classical,EGM, ChenHou1,EP,EP2} and is behind the computer assisted proof in \cite{ChenHou2,ChenHou3} in addition to guiding the numerical effort \cite{Wang}. We will now proceed to explain this perturbative argument, which first appeared in this context in \cite{EJDG}. 

\subsection{Perturbative Argument: A Guide}
We will now address the following question.
\begin{question}
Suppose that we have an approximate solution to $(\Omega_*,\lambda_*)$ to \eqref{SS1}-\eqref{SS2} with error $\varepsilon,$ which is qualitatively small. When can we say that we have a true solution to \eqref{SS1}-\eqref{SS2}?
\end{question}
Clearly, the main question relates to the invertibility of the linear operator $\mathcal{L}_*$ defined by
\[\mathcal{L}_*(\Omega)=\Omega+\lambda_*x\cdot\nabla\Omega +U_*\cdot\nabla\Omega+\nabla\Omega_*\nabla\times(-\Delta)^{-1}\Omega-\nabla U_*\Omega-\Big(\nabla\nabla\times(-\Delta)^{-1}\Omega\Big)\, \Omega_*\]
\[+\lambda x\cdot\nabla\Omega_*.\]Notice that $\lambda$ is a \emph{free} parameter that can be used to improve the invertibility properties of $\mathcal{L}_{*}(\Omega).$ In particular, we should view the $\lambda$ term as a rank-one term that we are free to choose. We have written it on its own line to emphasize this point. 

We are now going to "answer" the question positively by imposing two informal hypotheses. We want to clarify that these hypotheses are merely written this way for pedagogical purposes. For this reason, we will completely do away with a very important technical part: the construction of the relevant spaces, which the author views as a secondary issue to the main conceptual points that the hypotheses should convey. 

The first hypothesis relates to the invertibility of $\mathcal{L}_*$ and it was the main idea of the paper \cite{EJDG}.
\begin{hypothesis}
Assume that there exists a choice of the functional $\Omega\rightarrow\lambda$ for which $\mathcal{L}_*$ is an \emph{invertible} operator.  
\end{hypothesis}
It is not difficult to check that, if we have a true solution $\Omega_*$, scaling invariance gives us that $x\cdot\nabla\Omega_*$ is automatically in the kernel of $\mathcal{L}_*$ with $\lambda=0$. Choosing $\lambda$ properly should do away with this issue.  \noindent Next, we state the second main hypothesis.
\begin{hypothesis}
$\lambda_*>0$ and the vector field $\tilde U_*:=U_*+\lambda_*x$ is {\bf outgoing}. One way to quantify this is by \[\tilde U_*(x)\cdot x\geq c |x|^2,\] for all $x,$ and some $c>0.$  
\end{hypothesis}
\noindent What this means heuristically is that information is being pushed out of zero and into the bulk. In particular, if we are considering functions that vanish very strongly near zero, they are being damped very strongly by the joint effect of $\lambda x\cdot\nabla+U_*\cdot\nabla.$
What this means technically is that the main part of the linear operator, namely the transport part, is coercive on weighted spaces. Note that the second Hypothesis is usually quite easy to check. The first is the one that may be difficult to carry out analytically \cite{EP2}.

We now state a "meta theorem," which of course requires extra case-dependent hypotheses to be useful and true. 

\vspace{2mm}

\noindent {\bf Meta Theorem.} \emph{
Assume the two hypotheses hold for some $(\Omega_*,\lambda_*)$ and that the $\varepsilon$ is small enough. Then, there exists a true solution to \eqref{SS1}-\eqref{SS2}.}

\vspace{2mm} 

\begin{remark}
Small enough, here, depends of course on the function space, the size of $\Omega_*$ and $\lambda_*$, the norm of $\mathcal{L}_*^{-1},$ and a lower bound on $c$ in the second Hypothesis. We also note that, a brief look at \eqref{SS1}-\eqref{SS2} reveals that a "good" approximate solution must satisfy $\Omega_*\sim |x|^{-1/\lambda_*}$ when $|x|\rightarrow\infty,$ which we tacitly assume to be the case. 
\end{remark}

The use of the hypothesis on the invertibility of $\mathcal{L}_*$ is obvious, but we should explain where the second hypothesis comes from. Indeed, when solving a nonlinear problem of the form:
\[\mathcal{L}_*(\Omega)=\varepsilon +N(\Omega),\] we will have a unique solution if $N$ is a bounded (possibly nonlinear) operator so long as $\mathcal{L}_*^{-1}$ is bounded and $\varepsilon$ is small enough. In our case, of course, $N$ is not a bounded operator (it contains the term $U\cdot\nabla\Omega$). Therefore, we cannot apply Picard iteration to get existence. However, if we can rewrite:
\[\mathcal{L}_*=\mathcal{L}_{**}+K,\] with $K$ a smoothing or even finite-rank operator and $\mathcal{L}_{**}$ \emph{coercive}, it turns out that we can get back an inverse function theorem. The idea of doing such a decomposition seems to have appeared in many previous works on self-similar singularities. An example of the use of this idea to prove an inverse function theorem is given in Section 7 of \cite{EP}, inspired by the work \cite{MRRS}. 

Let us now state a very simple fact in 1d that illustrates the use of the second hypothesis in getting such a decomposition of $\mathcal{L}_*.$ 
\begin{lemma}
Fix $u:[0,1]\rightarrow \mathbb{R}$ with $u(0)=u(1)=0$ and $g:[0,1]\rightarrow\mathbb{R}$ are smooth.
Consider the operator $\mathcal{L}_*$ defined by:
\[\mathcal{L}_*(f)=u(x)\partial_xf+g(x) f(x).\] Assume that $u'(0)>0$, $u>0$ on $[0,1],$ and that $g(1)>0.$ Then, the operator $\mathcal{L}_*$ can be written as the sum of a coercive part and a finite rank part on a suitably chosen weighted Sobolev space. 
\end{lemma}
\noindent The idea behind the proof of this lemma is the key to many of the coercivity estimates done in the existing papers on self-similar singularities in the Euler equation and related models. The transport part of the operator $\mathcal{L}_*$ pushes information from $x=0$ to $x=1$. Moreover, the second piece of the operator is "positive" near $x=1.$ From a technical point of view, the first step in the proof is to observe that, if we consider $f$ vanishing very strongly at $x=0$ and we compute:
\[\int_0^\delta \mathcal{L}_*(f) f \frac{1}{x^N} dx\geq c \int_0^\delta \frac{f(x)^2}{x^N},\] so long as $N$ is sufficiently large and $\delta$ is sufficiently small. This gives coercivity in a small region around $x=0.$ It is then easy to get coercivity in a similar way on $[\frac{\delta}{2},1-\delta].$ On $[1-2\delta,1],$ we use the positivity of $g.$ We leave the details to the interested reader.
\section*{Acknowledgements}
The author is indebted to all his mentors, collaborators, and friends for sharing their knowledge with him. These begin with M. B. Elgindi, and then P. Constantin, F.H. Lin, N. Masmoudi, and P. Rabinowitz, as well as R. Beekie, R. Bianchini, M. Coti-Zelati, T. Drivas, T. Ghoul, D. Ginsberg, S. Ibrahim, I. Jeong, M. Jo, A. Kiselev, K. Liss, F. Pasqualotto, A. Said, D. Sullivan, V. \v{S}ver\'ak, E. Titi, K. Widmayer, and many more. He particularly thanks T. Drivas and I. Jeong for many years of collaboration and friendship without which many of the results presented would not have been possible. 


\end{document}